\documentclass[a4paper,12pt]{amsart}
\usepackage{amsmath}
 \usepackage{latexsym, amsthm, amscd, euscript, float, subfig}
\usepackage{amssymb}
\usepackage{ifthen}
\usepackage{graphicx}
\usepackage{float}
\usepackage[utf8]{inputenc}
\usepackage{cite}
\usepackage{amsfonts}
\usepackage{amscd}
\usepackage{amsxtra}
\usepackage{color}
\usepackage{mathtools}

\newtheorem{theorem}{Theorem}[section]
\newtheorem{lemma}[theorem]{Lemma}
\newtheorem{corollary}[theorem]{Corollary}

\newtheorem{definition}[theorem]{Definition}

\theoremstyle{definition}
\newtheorem{remark}[theorem]{Remark}

\newtheorem{question}[theorem]{Question}

\numberwithin{equation}{section}

\def\be{\begin{equation}}
\def\ee{\end{equation}}

\def\diam{\text{ diam}}
\def\dist{\text{ dist}}
\def\R{\mathbb{R}}
\newcounter{alphabet}

\makeatletter
\newcommand*{\rom}[1]{\expandafter\@slowromancap\romannumeral #1@}
\makeatother
\begin{document}
\bibliographystyle{amsplain}
\title{On invariance of John domains under quasisymmetric mappings}
\author[Vasudevarao Allu]{Vasudevarao Allu}
\address{Vasudevarao Allu, School of Basic Sciences, 
Indian Institute of Technology Bhubaneswar,
Bhubaneswar-752050, 
Odisha, India.}
\email{avrao@iitbbs.ac.in}
\author[Alan P Jose]{Alan P Jose}
\address{Alan P Jose, School of Basic Sciences, Indian Institute of Technology Bhubaneswar,
Bhubaneswar-752050, Odisha, India.}
\email{alanpjose@gmail.com}

\begin{abstract}
In this paper, we prove that if a homeomorphism is quasisymmetric relative to the boundary of the domain then it maps a length John domain to a diameter John domain. 
 Moreover, we prove a necessary and sufficient condition for a diameter John domain to be length John and thereby prove that if $f:G\rightarrow G'$ is $(M, C)-$CQH map, where $G$ is a John domain, and the map extends to the boundary such that the extension is $\eta-$QS relative to $\delta G$ then $G'$ is a John domain.
 In addition, we  characterize  distance John domains using weak minimizing property.
\end{abstract}

\maketitle

\section{Introduction}

The concept of quasiconformality in the plane was introduced by Grot\"{o}zsch in 1928. The idea was later extended by a number of authors to the case where $\mathbb{R}^n \text{ for } n\geq 3$. 
Later, Gehring and his students Palka and Osgood \cite{gehpal,gehosg} have introduced the quasihyperbolic metric which is the generalization of the classical hyperbolic metric in the unit disk.
Since then the quasihyperbolic metric has become an important tool in the studying the quasiconformal maps in Euclidean spaces. 
In the 1990s, V{\"a}is{\"a}l{\"a}, through a series of papers, see \cite{va1, va2, va3} , began his systematic investigation of quasiconformality in Banach spaces of infinite dimension, which he termed the free quasiworld.
Since in infinite dimensional Banach spaces many of the methods and tools that are helpful in $\mathbb{R}^n$ are unavailable, he constructed the theory using just elementary arguments. Recently, there have been many enquiries towards this subject (see \cite{zhousamy, zhou-li-rasila-2022, zhou-samy-2023, rasila-talponen-2014}).
\vspace*{2mm}

John domain were introduced by F. John \cite{john} in 1961 in connection with elasticity theory. Later, in 1978 O. Martio and J. Sarvas \cite{ms} coined the term "John domain" and used it to define uniform domains.
 These two classes were used to study injectivity properties of locally quasiconformal and locally bilipschitz maps. 
 The properties and characterizations of John domains in Euclidean spaces and metric spaces were explored by several authors (see \cite{kim-lang-1998, gehring-hag-martio-1989, guo-koskela-2014, heinonen-1989, herron-1999}). An excellent exposition on John disks is provided by N{\"a}kki and V{\"a}is{\"a}l{\"a}\cite{nakki-vaisala-1991}.
\vspace*{2mm}

In this paper, $E$ and $E'$ will always denote real Banach spaces of dimension atleast two. Furthermore, $G \subset E$, $G' \subset E'$ will always be domains (open connected sets). For notations and definitions we refer to \cite{quasiworld}.
In \cite{va2}, V{\"a}is{\"a}l{\"a} obtained the following result  relating the class of uniform domains and quasim\"{o}bius maps in Banach spaces, which is the generalization of the corresponding result in $\mathbb{R}^n$, see \cite{vaqm}.

\begin{theorem}[\cite{va2}]\label{va2result}
Suppose that $G$ is a $c$-uniform domain and that $f:G\rightarrow G'$ is freely $\psi$-quasiconformal, briefly $\psi$-FQC, where $G$ and $G'$ are domains in the Banach spaces $E$ and $E'$, respectively. Then the following conditions are quantitatively equivalent:
\begin{enumerate}
\item $G'$ is $c_1$-uniform,
\item $f$ is $\eta $-quasim\"{o}bius.
\end{enumerate}
\end{theorem}
 Later many enquired whether the result holds if you replace $\psi$-FQC maps with $C-$coarsely $M-$quasihyperbolic, briefly $(M,C)$-CQH,  and $\eta $-quasim\"{o}bius with $\eta $-quasim\"{o}bius rel $\delta G$  \cite{hlvw}
or $\eta $-quasim\"{o}bius on $\delta G$ \cite{zhousamy}. \\
In 2021, Zhou and Rasila \cite{zhourasila} considered the following question:
\begin{question} \label{rasilaqn}
Suppose that  $G$ is a $c$-uniform domain and that a homeomorphism $f:\bar{G} \rightarrow \bar{G'} $ is $\eta$ quasim\"{o}bius relative to $\delta G$ and maps $G$ onto $G'$. Is $G'$ a $c'$ - uniform with $c'=c'(c,\eta)$?
\end{question}
Studying this problem, Zhou and Rasila \cite{zhourasila} have obtained the following result concerning the implications of uniform domains, diameter uniform domains and min - max property:
\begin{theorem}\label{rasilathm}
Suppose that $G\subset E$ is a domain. Then we have the following implications: $G$ is $c$-uniform $\implies$ $G$ satisfies the min-max property $\implies $ $G$ is diameter $c_1$ uniform $\iff $ $G$ is $\delta$ -uniform for some $0<\delta <1.$
\end{theorem}

Zhou and Rasila \cite{zhourasila} have obtained the following result as a consequence of Theorem \ref{rasilathm}.

\begin{theorem}\label{rasilainvthm}
Let $G$ and $G'$ be proper domains in Banach spaces $E$ and $E'$, respectively. Suppose that $G$ is $c$-uniform and that a homeomorphism $f:\bar{G} \rightarrow \bar{G'}$ is $\eta$-quasim\"{o}bius relative to $\delta G$ and maps $G$ onto $G'$. Then $G'$ is diameter $c'$-uniform with some $c'=c'(c,\eta)$.
\end{theorem}

Also in 2022, Ouyang, Rasila and Guan \cite{org} characterized diameter uniform domains using weak min-max property and showed that diameter uniform domains are invariant under quasim\"{o}bius maps relative to the boundary. Since every uniform domain is a John domain it is natural to ask whether Theorem \ref{rasilainvthm} holds if we replace uniform domains by John domains.
 
But if  $G = \mathbb{D} \setminus [-\frac{1}{2}, 1)$, where $\mathbb{D} = \left\{z\in \mathbb{C} : |z| < 1\right\}$ and $\left[-\frac{1}{2}, 1\right)$ denotes the line segment joining $-\frac{1}{2}$ and $1$,  and $f:G \rightarrow \mathbb{C} $ defined by $f(z) = z/|z|^2$, then $f$ is $\eta $ quasim\"{o}bius with $\eta(t)=81t$. Nevertheless, $f$ maps a John domain onto a domain which is not John.

Therefore, it is reasonable to ask the following question:

\begin{question}\label{qn}
Let $E$ and $E'$ be Banach spaces of dimension $n\geq 2$ and $G\subset E$, $G'\subset E'$ be domains. Further, suppose that $G$ is a $c$-John domain. If a homeomorphism $f:G \rightarrow G'$, what are the conditions $f$ need to satisfy so that $G'$ is a $c'$-John domain.
\end{question}

In 2019, Li, Vuorinen and Zhou \cite{vuorinen} characterized John domain with a center in a locally $(\lambda, c )$-quasiconvex, rectifiably connected, non-complete metric space and thereby obtained that if $f$ is an $\eta$-quasisymmetric homeomorphism then it maps  John domain with center $x_0$ to a John domain with center $f(x_0)$.

\begin{theorem}\label{vuoresult}
Suppose  $f:D \rightarrow D'$ is an $\eta$- quasisymmetric homeomorphism between two locally $(\lambda , c) $- quasiconvex, rectifiably connected, non-complete metric spaces. If $D$ is a length $a$-John space with center $x_0$, then $D'$ is a length $a'$ -John space with center $x_0'=f(x_0)$, where $a'$ depends on $\lambda , c, \eta $ and $a$.
\end{theorem}
Now we introduce minimizing and weak minimizing property for a domain $G\subset E$.

\begin{definition}\label{minppty}
Let $G\subset E $ be a domain then we say $G$ has minimizing property if there exists a family of curves $\Gamma$ in $G$ and a constant $c$ such that each pair of points in $G$ can be joined by a curve $\gamma \in \Gamma$  satisfying
\begin{equation}\label{mineqn}
\min_{j=1,2}|x_j -y| \leq c|x-y|
\end{equation} 
for each ordered triple of points $x_1, x, x_2 \in \gamma $ and each $y \in \delta G$.
\end{definition}

\begin{definition}\label{weakminppty}
Let $G$ be a domain if there exists a family of curves $\Gamma$ in $G$ and a constant $c$ such that for each $x_1, x_2 \in G$ there exists a curve $\gamma :x_1 \curvearrowright x_2$ in $\Gamma$ such that  
\begin{equation}\label{weakmineqn}
\min_{j=1,2}|x_j -y| \leq c|x-y|
\end{equation} 
for each $x \in \gamma$ and $y\in \delta G$ then we say that $G$ has weak minimizing property.
\end{definition}
\begin{remark}
In 1987, Gehring and Hag have proved that the endpoints of a hyperbolic geodesic $\gamma $ in $\mathbb{B}^n$ both minimize and maximize the distance from a point of $\gamma $ to any point outside $\mathbb{B}^n$ upto the constant $2^{1/2}$ and the constant is the best possible. 
 They called it the min-max property of hyperbolic geodesics and motivated by this, they introduced domains with min-max property in $\R^n$. 
 It is worth to mention that O.J. Broch \cite{} in his thesis proved that minimizing property  characterizes John domains in $\mathbb{R}^n$.
\end{remark}
Now we state our main results. (See subsection \ref{subjohn} for definitions.)

\begin{theorem}\label{thm1}
Let $G\subset E$ be a domain then the following implications  hold:\\
$G$ is $c$-John domain $\implies $ $G$ has the minimizing property $\iff $ G is diameter $c_1$ John.
\end{theorem}
As a consequence, we partially answer Question \ref{qn} as follows:
\begin{theorem}\label{thm2}
Let $G$ and $G'$ be proper domains in Banach spaces $E$ and $E'$, respectively. Suppose that $G$ is $c$-John domain and that $f:\bar{G} \rightarrow \bar{G'}$ is a homeomorphism that maps $G$ onto $G'$ which is $\eta $ quasisymmetric relative to $\delta G$ (see Definition \ref{qsdefn}). 
Then $G'$ is diameter $c'$-John domain with $c'=c'(c,\eta)$.
\end{theorem}

We prove that the weak minimizing property characterizes distance John domains. Note that the class of diameter John domains is properly contained in the class of distance John domain, see \cite{nakki-vaisala-1991}.
\begin{theorem}\label{thm3}
Let $G\subset E$ be a domain, where $E$ is a Banach space of dimension $n\geq 2$. 
Then $G$ is distance $c$-John domain if, and only if $G$ satisfies the weak minimizing property.
\end{theorem}

Similar to Theorem \ref{thm2} we get the following result corresponding to distance John domains.

\begin{corollary}\label{thm4}
If $f:\bar{G} \rightarrow \bar{G'}$ is $\eta$ quasisymmetric relative to $\delta G$, maps $G$ onto $G'$ and $G$ is distance $c$- John domain then $G'$ is distance $c'$- John domain, where $c'=c'(c,\eta)$.
\end{corollary}

Since, there are diameter John domains which are not length John, it is reasonable to ask when does a diameter John domain is length John. We answer this as follows:
\begin{theorem}\label{thm5}
Let $G\subsetneq E$ be a $c-$ John domain then every pair of points can be joined by an arc $\gamma $ which satisfies the $\psi-$natural condition. 
Conversely, if $G\subsetneq E$ is a diameter $c-$John domain such that any two points in $G$ can be joined by diameter $c-$John arc which satisfies the $\psi-$natural condition. 
Then $G$ is length $c'-$John for some constant $c'=c'(c,\psi)$.
\end{theorem}
As an immediate consequence of Theorem \ref{thm2} and Theorem \ref{thm5}, we obtain the following result:
\begin{corollary} \label{cor1}
Let $G$ be a $c-$John domain and $f:\bar{G} \rightarrow \bar{G'}$ be an $\eta-$quasisymmetric relative to $\delta G$ which maps $G$ to $G'$. Further, if $G'$ is $\psi-$natural then, $G'$ is $c'-$John, where $c'=(c, \eta, \psi)$. 
\end{corollary}
As an application of Theorem \ref{thm5}, we answer Question \ref{qn} as follows.
\begin{theorem} \label{thm6}
Let $G$ be a length John domain and that $f:G \rightarrow G' $ is $(M, C)-$CQH. 
If $f$ extends to a homeomorphism $\bar{f}:\bar{G} \rightarrow \bar{G'}$ which is $\eta-$quasisymmetric relative to $\delta G$, then $G'$ is length John domain. 
\end{theorem}
Since every quasisymmetric map $f:G \rightarrow G'$ can be extended to the boundary such that  $\bar{f}:\bar{G} \rightarrow \bar{G'}$ is quasisymmetric (see \cite[Theorem 6.12]{quasiworld})and every quasisymmetric map is $(M, C)-$CQH (see  \cite[Theorem 7.17]{quasiworld},  \cite[Theorem 4.14]{va2}) Theorem \ref{thm6} is an improvement of the result of Li, Vuorinen, and Zhou \cite[Theorem 1.2]{vuorinen}.
%Since, for a bounded John domain in $\mathbb{R}^n$ every distance John domain are length John domain we have the following easy consequence of Corollary \ref{thm4}. 
%\begin{corollary}\label{cor}
%If $D \subset \mathbb{R}^n $ is a  John domain in the sense Definition \ref{bddjohn}  and $f:\bar{D} \rightarrow \bar{D'} \subset \mathbb{R}^n$ is $\eta$-QS rel $\delta D$, then $D'$ is also John domain as in Definition \ref{bddjohn}.
%\end{corollary}
%In \cite[Corollary 3.2]{vuorinen}, Li, Vuorinen, and Zhou have shown that John domains in $\mathbb{R}^n$ are preserved under quasisymmetric homeomorphisms with a dimension free control function. Since, every quasisymmetric map defined on a domain in a Banach space extends to the boundary of the domain,(see \cite[Theorem 6.12]{quasiworld}), Corollary \ref{cor} is an improvement the result of Li, Vuorinen, and Zhou \cite{vuorinen}.
 
\section{Preliminary Results}
In this section we present the basic definitions and preliminary results.

\subsection{Notations}
The norm of a vector $x$ in $E$ is denoted as $|x|$ and $[x,y]$ denotes the closed line segment with endpoints $x$ and $y$.\\[2mm]
For a set $D\subset E$, $\bar{D}$ denotes the closure of $D$ and $\delta D = \bar{D}\setminus D$ be its boundary in the norm topology. 
For an element $x\in D$, let $\delta_D(x)$ denote the distance of the point $x$ to the boundary of $D$, {\it i.e.,} $\delta_D (x) = \inf \left\{|x-y| : y\in \delta D\right\}$. For a bounded set $D$ in $E$, diam($D$) is the diameter of  $D$. 
A  curve is a continuous function $\gamma : [a,b] \rightarrow E$. 
The length of $\gamma$ is defined by
$$ l(\gamma ) = \sup \left\{\sum_{i=1}^n |\gamma (t_i)- \gamma (t_{i-1}| \right\},$$
where the supremum is taken over all partitions $ a=t_0 <t_1< \cdots < t_n=b.$

\subsection{John domains} \label{subjohn}
\begin{definition}\label{john}
Let $G \subsetneq E$ be a domain and let $c\geq 1$. Then $G$ is called $c$-John, if each pair of points $x_1, x_2$ in $G$ can be joined by a curve $\gamma$ in $G$ satisfying the following:
\begin{equation}\label{johncondition}
\min_{j=1,2} l(\gamma[x_j,x]) \leq c \delta_G (x)\text{ for all } x\in \gamma ,
\end{equation}

where $\gamma[x_j,x]$ denotes the part of the curve $\gamma$ from $x_j$ to $x$. We say that such a curve is a $c$-cone arc. This definition is known as the cigar definition of John domains.

\end{definition}
If the length of the curve in \eqref{johncondition} in the Definition \ref{johncondition} is replaced by diam$\left(\gamma[x_j,x]\right)$ then the domain is said to be diameter $c$- John domain. We say that $G$ is distance $c$-John if \eqref{johncondition} is replaced by
$$\min_{j=1,2} |x_j - x| \leq c \delta_G (x),$$
for all $x \in \gamma $.

\begin{definition}\label{bddjohn}
For $c\geq 1$, a non complete metric space $(D,d)$ is called $c$-John domain, if there is $x_0 \in D$, called the John center of $D$, such that every point $x$ in $D$ can be joined to $x_0$ by a curve $\alpha $ satisfying 
$$ l\left(\alpha[x,z]\right) \leq c \delta_D (z),$$
for each $z \in \alpha $, and we say that $\alpha $ is a $c$-carrot arc.
\end{definition}

We define diameter $c$-John with center and distance $c$-John with center similarly as we have defined for general John domains.

\begin{remark}
It is easy to see that if $D$ is a John domain with center $x_0$ then $D\subset B(x_0, c\delta_D(x_0)$. In particular every John domain with center $x_0$ is bounded. If the domain $D \subset \mathbb{R}^n$ is bounded then both definitions \ref{john} and   \ref{bddjohn} are quantitatively equivalent, (see \cite{vaisalaunif}). In this article, we  focus primarily on domains that are John in the cigar sense. 
\\[2mm]
It is easy to see that length $c$-John $\implies$ diameter $c$-John $\implies$ distance $c$-John. 
In $\mathbb{R}^n$, Martio and Sarvas \cite{ms} have shown that for a  John domain $D \subset \mathbb{R}^n$ with center $x_0$ diameter $c$-John implies length $c$-John. In 1991,N{\"a}kki and V{\"a}is{\"a}l{\"a} \cite{nakki-vaisala-1991} have showed that this result still holds in the case of unbounded domains in $\R^n$.
In \cite{vaisalaunif}, V{\"a}is{\"a}l{\"a} proved that for a bounded domain $D \subset \mathbb{R}^n$, diameter $c$-John and distance $c$-John are quantitatively equivalent. 
N{\"a}kki and V{\"a}is{\"a}l{\"a} \cite{nakki-vaisala-1991} showed by an example that this is not true in unbounded domains.
V{\"a}is{\"a}l{\"a} \cite{vaisala-2004} constructed an example of a diameter John domain in infinite dimensional Hilbert space which is not a length John domain.
\end{remark}

\subsection{Quasihyperbolic metric}
Quasihyperbolic metric was first introduced by Gehring and Palka \cite{gehpal} for proper domains in $\R^n$. It is a very crucial tool in studying quasiconformality in Banach spaces.
\begin{definition}
Let $G\subsetneq E$ be a domain in a Banach space and let $\gamma $ be a rectifiable arc then the quasihyperbolic length of $\gamma $ is defined as the number(see \cite{va1})
$$l_k(\gamma) = \int_{\gamma} \frac{|dz|}{d_G(z)}.$$
For each pair of points $x,y\in G$ the quasihyperbolic distance denoted by $$k_G(x, y) = \inf_{\gamma} l_k(\gamma),$$
where the infimum runs over all arcs $\gamma$ joining $x$ and $y$.
\end{definition}
We shall usually abbreviate $d(z)=d_G(z)$, $k(x, y)= k_G(x, y)$ and $k'(x, y)= k_{G'}(x, y)$ where $G'\subsetneq E$ is a domain.
We also recall some of the elementary inequalities regarding quasihyperbolic metric:
\be \label{qh1}
\left|\log\left(\frac{d(x)}{d(y)}\right)\right| \leq k(x,y) \text{ for every } x,y \text{ in } G.
\ee
More generally, we have
\be \label{qh2}
\log\left(1+ \frac{l(\gamma)}{\min\{d(x), d(y)\}}\right) \leq l_k(\gamma).
\ee
\begin{definition}
Let $\psi :[0, \infty) \rightarrow [0, \infty)$ be an increasing function. We say that a curve $\gamma$ satisfies the $\psi-$natural condition if the following holds:
\be \label{nat}
\diam_k(\gamma[x,y]) \leq \psi\left( \frac{\diam(\gamma[x,y])}{\dist(\gamma[x,y], \delta G)}\right) \text{ for every } x,y \in \gamma
\ee
where $\diam_k$ means the diameter under the quasihyperbolic metric.\\
A domain $G\subsetneq E$ is said to be $\psi-$natural if \eqref{nat} holds for every connected subset $A\subset G$.
\end{definition}
We remark that the class of natural domains is fairly large, in fact each domain in a finite dimensional Banach space is $\psi-$natural with $\psi$ depending only on the dimension (see\cite{va3}). 
V{\"a}is{\"a}l{\"a} has constructed an example of a broken tube domain in infinite dimensional Hilbert space which is not natural (see \cite{vaisala-2004}).
\subsection{Maps in freequasiworld}
Let $X$ be a metric space. By a triple in  $X$  we mean an ordered sequence $T=(x,y,z)$ of three distinct points in $X$. The ratio of $T$ is the number
$$\rho (T) = \frac{|y-x|}{|z-x|}.$$
If $f:X \rightarrow Y$ is an injective map, the image of a triple $T=(x,y,z)$ is the triple $f(T)=(f(x),f(y),f(z)).$ Suppose that $A\subset X$. A triple $T=(x,y,z)$ in $X$ is said to be a triple in the pair $(X,A)$ if $x\in A$ or if $\{y, z\} \subset A$. 
Equivalently, both $|y-x|$ and $|z-x|$ are distances from a point in $A$.

\begin{definition} \label{qsdefn}
Let $X$ and $Y$ be metric spaces an embedding $f:X \rightarrow Y$ is said to be $\eta$-quasisymmetric if there is a homeomorphism $\eta: [0, \infty ) \rightarrow [0, \infty )$ such that
\begin{equation}\label{qs}
\rho(fT) \leq \eta(\rho(T))
\end{equation}
holds for each triple $T$ in $X$.\\[2mm]
Let $A \subset X$, then $f$ is said to be $\eta $ -quasisymmetric relative to $A$, abbreviated $\eta$-QS rel $A$, if the condition \eqref{qs} holds for every triple $T$ in $(X,A)$. 
Thus quasisymmetry rel $X$ is equivalent to ordinary quasisymmetry.
\end{definition}
Note that an embedding $f:X \rightarrow Y$ is $\eta $-QS rel $A$ if, and only if, $\rho(T) \leq t$ implies that $\rho(f(T)) \leq \eta (t)$ for each triple $T$ in $(X, A)$ and $t\geq 0$.\\[2mm]
In 1999, V\"{a}is\"{a}l\"{a} \cite{quasiworld} proved  that \eqref{qs} implies $f$ is an embedding unless $f$ is a constant.
\begin{definition}
A homeomorphism $f:G \rightarrow G'$ between domains $G\subsetneq E$ and $G'\subsetneq E'$, where $E$ and $E'$ are Banach spaces, is $C-$coarsely $M-$quasihyperbolic if 
$$\frac{k(x, y)-C}{M} \leq k'(f(x), f(y)) \leq M k(x, y)+C \text{ for all } x, y \in G.$$
\end{definition}
\section{Main Results}\label{sec3}

The proof of Theorem \ref{thm1} is divided into the following two lemmas.

\begin{lemma}\label{lemma1}
Suppose $G\subset E$ is a domain, and if $\gamma \subset G$ is a $c-$cone arc, then $\gamma$ satisfies \eqref{mineqn} with constant $c'=1+c$.
\end{lemma}

\begin{proof}
%Let
%\begin{equation*}
%\Gamma \coloneqq \left\{\gamma : x \curvearrowright y : \gamma \text{ is a $c$-cone arc joining $x$ and $y$ } \right\}.
%\end{equation*}
%Since $G$ is $c$-John, $\Gamma $ is not empty. 
%Now, for $u,v \in G$ there exists a $c$-cone arc $\gamma \in \Gamma$ connecting $u$ and $v$. 
Let  $\gamma$ be a $c$-cone arc connecting $u$ and $v$. For an ordered triple of points $x_1, x, x_2 \in \gamma$ and $y\in \delta G$, it suffices to verify \eqref{mineqn} for these points.
It is easy to see that the sub-arc  $\gamma[x_1, x_2]$ is again a $c$ -cone arc.
%Let $w \in \gamma $ be such that $l(\gamma [u,w]) = l\left(\gamma [w, v]\right) .$
\\[2mm]
%We divide this into three cases according to the position of the points $x_1$ and $x_2$ in $\gamma $ with respect to $w$.\\
%{\it Case \rom{1}}: If $x_1, x_2 \in \gamma [u,w],$ then
%\begin{equation*}
%\min \left\{ l\left(\gamma[x_1,x]\right), l\left(\gamma[x,x_2]\right)\right\} = l\left(\gamma[x_1,x]\right) \leq l\left(\gamma[u,x]\right) \leq c \delta_G(x).
%\end{equation*}
%
%{\it Case \rom{2}}: If $x_1, x_2 \in \gamma [w,v],$ then 
%\begin{equation*}
%\min \left\{ l\left(\gamma[x_1,x]\right), l\left(\gamma[x,x_2]\right)\right\} = l\left(\gamma[x,x_2]\right) \leq l\left(\gamma[x,v]\right) \leq c \delta_G(x).
%\end{equation*}
%
%{\it Case \rom{3}}: $\gamma [x_1,x_2] = \gamma [x_1,w] \cup \gamma [w,x_2]$
%
%If $x\in \gamma [x_1,w]$, then similarly as in Case \rom{1} we obtain
%\begin{equation*}
%\min \left\{ l\left(\gamma[x_1,x]\right), l\left(\gamma[x,x_2]\right)\right\}  \leq c \delta_G(x).
%\end{equation*}
%
%If $x\in \gamma [w,x_2]$, then similarly as in Case \rom{2} we obtain
%\begin{equation*}
%\min \left\{ l\left(\gamma[x_1,x]\right), l\left(\gamma[x,x_2]\right)\right\} \leq c \delta_G(x).
%\end{equation*}
Therefore, we have, 
\begin{equation*}
\min_{j=1,2} |x_j -x| \leq \min_{j=1,2} l\left(\gamma[x_j , x]\right) \leq c \delta_G(x) \leq c|x-y|.
\end{equation*}
Thus, we obtain 

\be
\min_{j=1,2} |x_j -y| \leq \min_{j=1,2} |x_j -x| +|x-y| \leq (1+c) |x-y|.
\ee

\end{proof}

\begin{lemma}\label{lemma2}
Suppose that  $G \subset E $ is a domain and $\gamma \subset G$ be a curve which satisfies \eqref{mineqn} with some constant $c$, then $\gamma$ is diameter $c'$- John arc, where $c'=c'(c)$. Conversely, if $\gamma$ is a diameter $c-$ John arc then $\gamma$, then $\gamma$ satisfies \eqref{mineqn} with the constant $1+c$.
\end{lemma}
\begin{proof}
 Let $x_1, x_2 \in G$ and  $\gamma$ be a curve joining $x_1$ and $x_2$  satisfying \eqref{mineqn}. Let $x\in \gamma $, we choose a point $z\in \delta G$ such that $|x-z| \leq 2 \delta_G (x).$
 For $j=1$ or $j=2$, we claim that 
 \be \label{claim}
  \gamma[x_j,x] \subseteq B\left(z, 2c\delta_G(x)\right).
  \ee
  If it is not the case, then there exists $u_j \in \gamma[x_j, x]$ such that $|u_j -z |> 2c\delta_G (x)$ for $j=1,2,$
which implies that
$$c|x-z| < \min_{j=1,2} |u_j -z|,$$
which contradicts \eqref{minppty}.
From \eqref{claim}, it follows that 
$$ \min_{j=1,2} \text{diam }( \gamma[x_j,x]) \leq c' \delta_G (x),$$
where $c' = 4c$. Hence $\gamma $ is $c'-$cone arc.  \\
Conversely, if $\gamma$ is a diameter $c-$John arc joining $u, v\in G$,for an ordered triple of points $x_1, x, x_2 \in \gamma$ and $y\in \delta G$ the result follows from a similar argument as in Lemma \ref{lemma1}, since the subarc $\gamma[x_1, x_2]$ is again a diameter $c-$John arc.
\end{proof}
Now we are ready to prove Theorem \ref{thm1}.

\begin{proof}[Proof of Theorem \ref{thm1}]
Suppose that $G$ is a $c-$John domain.
Let
\begin{equation*}
\Gamma \coloneqq \left\{\gamma : x \curvearrowright y : \gamma \text{ is a $c$-cone arc joining $x$ and $y$ } \right\}.
\end{equation*}
Since $G$ is $c$-John, $\Gamma $ is not empty and for each $u,v \in G$ there exists a $c$-cone arc $\gamma \in \Gamma$ connecting $u$ and $v$. Then from Lemma \ref{lemma1} it follows that $\Gamma$ is the required family of curves and therefore $G$ has the minimizing property.\\
Now, suppose $G$ has the minimizing property with the family of curves $\Gamma$ and constant $c$. Then it follows from Lemma \ref{lemma2} that each $\gamma \in \Gamma$ is a diameter $c'-$John arc. Since for any $u, v\in G$ there exists a curve $\gamma \in \Gamma$ joining them we obtain that $G$ is diameter $c'$-John domain. \\
The reverse implication follows similarly from the arguments in the first implication. 

\end{proof}

\begin{proof}[Proof of Theorem \ref{thm2}]

Let $G$ and $G'$ be proper domains in Banach spaces $E$ and $E'$, respectively. Suppose $G$ has the minimizing property, $f:\bar{G} \rightarrow \bar{G'}$ is $\eta$ QS rel $\delta G$,  and maps $G$ onto $G'$. 
We show that $G'$ has the minimizing property.
Suppose $\Gamma $ is the family  of curves in $G$ such that for any two given points there exists a curve $\gamma \in \Gamma$ connecting them and satisfies \eqref{mineqn} with constant $c$ for any ordered triple of points $x_1, x, x_2 \in \gamma$ and $y\in \delta G$. We claim that $\Gamma ' = f(\Gamma)$ is the required family of curves in $G'$. \\[2mm]
Let $\gamma '$ be a curve in $\Gamma '$ and fix ordered triple of points   $x_1' = f(x_1), x'=f(x), x_2' = f(x_2)$ in  $\gamma '$ and $y'=f(y)$ be in $\delta G' $. Then $x_1, x, x_2 $ is an ordered triple of points in $\gamma$ and $y$ is in $\delta G$. Without loss of generality assume that $\min_{j=1,2} |x_j-y| = |x_1 - y|$. 
Then by the minimizing property of $G$ we obtain $|x_1 - y| \leq c |x-y|,$ which implies $|f(x_1)-f(y)| \leq \eta(c) |f(x) - f(y)|. $ 
Hence, $G'$ satisfies the minimizing property and Theorem \ref{thm2} follows from Theorem \ref{thm1}.
\end{proof}
\begin{remark} \label{remark}
Note that in the proof of Theorem \ref{thm2} we have actually proved that a map which is quasisymmetric relative to the boundary of a domain maps a curve with the minimizing property to a curve with the minimizing property with the constants depending only on the hypothesis. This together with Lemma \ref{lemma2} shows that the image of a diameter John arc under a quasisymmetric map relative to the boundary is again a diameter John arc with the constants depending only on the hypothesis. In particular, it implies that diameter John domains are invariant under quasisymmetric maps relative to the boundary. For Euclidean spaces, this leads to the following corollary.
\end{remark}
\begin{corollary}
If $D \subset \mathbb{R}^n $ is a  $c-$John domain and $f:\bar{D} \rightarrow \bar{D'} \subset \mathbb{R}^n$ is $\eta$-QS rel $\delta D$, then $D'$ is $c'-$ John domain with $c'=c'(c, \eta)$.
\end{corollary}

%Since, for a bounded John domain in $\mathbb{R}^n$ every distance John domain are length John domain we have the following easy consequence of Corollary \ref{thm4}. 
%\begin{corollary}\label{cor}
%If $D \subset \mathbb{R}^n $ is a  John domain in the sense Definition \ref{bddjohn}  and $f:\bar{D} \rightarrow \bar{D'} \subset \mathbb{R}^n$ is $\eta$-QS rel $\delta D$, then $D'$ is also John domain as in Definition \ref{bddjohn}.
%\end{corollary}
%In \cite[Corollary 3.2]{vuorinen}, Li, Vuorinen, and Zhou have shown that John domains in $\mathbb{R}^n$ are preserved under quasisymmetric homeomorphisms with a dimension free control function. Since, every quasisymmetric map defined on a domain in a Banach space extends to the boundary of the domain,(see \cite[Theorem 6.12]{quasiworld}), Corollary \ref{cor} is an improvement the result of Li, Vuorinen, and Zhou \cite{vuorinen}.
Now we provide the proof of Theorem \ref{thm3}.
\begin{proof}[Proof of Theorem \ref{thm3}]
Suppose that $G\subset E$ is a  distance $c$- John domain. Set 
$$ \Gamma = \left\{ \gamma: x_1 \curvearrowright x_2 : \gamma \text{ is the distance $c$-John arc } , x_1,x_2 \in G \right\}.$$
Fix $x_1,x_2 \in G$ and $\gamma: x_1 \curvearrowright x_2 $ in $\Gamma$. Then for $x\in \gamma$ and $y\in \delta G$, we have
\be 
\min_{j=1,2} |x_j - y| \leq \min_{j=1,2} |x_j - x| + |x-y| \leq c \delta_G (x) + |x-y| \leq (1+c) |x-y|.
\ee
Therefore, $G$ satisfies the weak minimizing property with the family of curves $\Gamma$ and constant $1+c$. 
\\[2mm]
Conversely, suppose $G$ has the minimizing property with $\Gamma '$ as the family of curves and $c'$ as the constant satisfying \eqref{weakmineqn}. Let $x_1,x_2 \in G$ and $\alpha : x_1 \curvearrowright x_2$ be the curve in $\Gamma '$.
For $x\in \alpha$, choose $z\in \delta G$ with $|x-z|\leq 2 \delta_G(x).$
Then by \eqref{weakmineqn}, we have

$$\min_{j=1,2} |x_j - x| \leq \min_{j=1,2} |x_j - z|+ |x-z| \leq (c+1)|z-x| \leq 2(c+1) \delta_G (x) .$$
Therefore, $G$ is distance $c'$- John, where $c'= 2(c+1)$.
\end{proof}
In light of Theorem \ref{thm3}, it suffices to prove the following lemma to prove Corollary \ref{thm4}.

\begin{lemma}
Suppose $G\subset E$ and $G' \subset E'$ be domains. Let $f: \bar{G} \rightarrow \bar{G'}$ be $\eta$-QS rel $\delta G$ and maps $G$ onto $G'$. Further, suppose that $G$ has the weak minimizing property, then $G'$ also has the weak minimizing property.
\end{lemma}
\begin{proof}
Fix $x_1', x_2' \in G'$ and $y' \in \delta G'$. Then there exist $x_1, x_2 \in G$, $y \in \delta G'$ such that $f(x_i)=x_i'$, for $i=1,2$ and $f(y)=y'$ . Further, there exists a curve $\gamma : x_1 \curvearrowright x_2 $ in $G$ and a constant $c$ such that for $x\in \gamma$,
$$ \min_{j=1,2}|x_j-y| \leq c|x-y|.$$
Without loss of generality, assume that $ \min_{j=1,2}|x_j-y| = |x_1 -y|$.
Thus by the quasisymmetry of $f$ we yield,
$$|x_1' - y'| \leq \eta(c) |f(x)-y'|.$$
Hence, it follows that $f(\gamma) $ is curve joining $x_1$ and $x_2$ satisfying \eqref{weakmineqn}.
Therefore, $G'$ has the weak minimizing property with constant $c'=\eta(c)$.
\end{proof}

\section{Proofs of Theorems \ref{thm5}, \ref{thm6} and Corollary \ref{cor1}}
We break down the proof of Theorem \ref{thm5} into two lemmas.
\begin{lemma}
Let $G\subsetneq E$ be a $c-$ John domain then every pair of points can be joined by an arc $\gamma $ which satisfies the $\psi-$natural condition.
\end{lemma}
\begin{proof}
Let $x_1, x_2 \in G$ and let $\gamma$ be the cone arc joining $x_1$ and $x_2$. We show that $\gamma $ satisfies the natural condition.
Let $x_0 \in \gamma $ be such that $l(\gamma [x_1, x_0]) = l(\gamma[x_2, x_0])$.
Set $\alpha_1 = \gamma [x_1, x_0]$ and $\alpha_2= \gamma[x_2, x_0]$.
\\[2mm]
If $\alpha_1  \subset \bar{B(x_1, d(x_1)/2)}$ then  we have
$$\frac{d(x_1}{2} \leq d(x_1)-|x_1 - x_0| \leq d(y) \leq d(x_1) +|x_0-x_1| \leq \frac{3}{2} d(x_1)$$
which yields
$$l_k(\alpha_1[x_1, x_0]) \leq \frac{2}{d(x_1)} l(\alpha_1[x_1, x_0]) \leq 2c \frac{d(x_0)}{d(x_1)} \leq 3c.$$
\\[2mm]
If $\alpha \nsubseteq \bar{B(x_1, d(x_1)/2)}$, then there is a point $w\in \alpha_1$ such that $l(\alpha_1[x_1, w]) = \frac{1}{2} d(x_1).$
Hence $l_k(\alpha_1[x_1, w] ) \leq 3c$.\\[2mm]
Next, if $\alpha_1 $is parametrized by arclength $s$ with $\alpha_1(0) = x_1$, then
$$s\leq l(\alpha_1 [x_1, z]) \leq c d(z)$$
and thus we obtain
$$ l_k \alpha_1[w,x_0]) = \int_{\alpha_1([w,y])} \frac{ds}{d(z)} \leq c \int_{\frac{d(x_1)}{2}}^{l(\alpha_1([x_1,x_0])} \frac{ds}{s} \leq c \log\left(\frac{2cd(x_0)}{d(x_1)} \right).$$
Therefore, we have
\begin{align}
    l_k \left(\alpha_1[x_1, x_0]\right) &= l_k\left(\alpha_1 [x_1, w]\right) + l_k \left(\alpha_1 [w, x_0]\right)  \nonumber \\
& \leq c \log \left(2c\frac{d(x_0)}{d(x_1)} \right) + 3c \nonumber \\
&= a \log \left( \frac{d(x_0)}{d(x_1)} \right) + b \label{nat1}
\end{align}
where $a =c$ and $b= c\log(2c)+3c $.\\
Similarly,
$$l_k(\alpha_2[x_2, x_0]) \leq a \log \left(\frac{d(x_0)}{d(x_2)}\right) + b.$$
Now from \eqref{nat1} we obtain,
\begin{align*}
\diam_k (\alpha_1[x_1, x_0]) & \leq l_k(\alpha_1[x_1, x_0]) \leq a \log \left(\frac{d(x_0)}{d(x_1)}\right) + b \\
& \leq a \log \left(1+ \frac{|x_1 -x_0|}{d(x_1)} \right) +b \\
& \leq a \log \left(1+ \frac{\diam (\alpha_1[x_1, x_0]}{\dist\left( \alpha_1, \delta G\right)} \right) +b
\end{align*}
which yields 
\be
\diam_k(\alpha_1[x_1,x_0]) \leq a \log \left( 1+ \frac{\diam(\gamma[x_1, x_2]}{\dist(\gamma, \delta G)} \right) +b
\ee
and similarly,
\be
\diam_k(\alpha_2[x_1,x_0]) \leq a \log \left( 1+ \frac{\diam(\gamma[x_1, x_2]}{\dist(\gamma, \delta G)} \right) +b.
\ee
Finally,
\begin{align*}
\diam_k(\gamma[x_1, x_2]) &\leq \diam_k(\alpha_1[x_1, x_0]) + \diam_k(\alpha_2[x_2,x_0]) \\
& 2c \log\left( 1+ \frac{\diam(\gamma[x_1, x_2]}{\dist(\gamma, \delta G)} \right) +2b.
\end{align*}
Now the proof follows from the fact that every sub-arc of a $c-$cone arc is again a $c-$cone arc.
\end{proof}

\begin{lemma}
Let $G\subsetneq E$ be a diameter $c-$John domain such that any two points in $G$ can be joined by diameter $c-$John arc which satisfies the $\psi-$natural condition. Then $G$ is length $c'-$John for some constant $c'=c'(c,\psi)$.
\end{lemma}
\begin{proof}
Let $x, y \in G$ and $\gamma $ be a diameter $c-$ John arc joining $x$ and $y$ satisfying the $\psi$-natural condition. Further, let $x_0$ be the midpoint of $\gamma$. That is, 
\be \label{big1}
\diam \left(\gamma[x, x_0]\right) = \diam \left(\gamma[x_0, y]\right) \leq c d(x_0).
\ee
We need to show that there is a $c'-$cone arc joining $x$ and $y$. By symmetry it suffices to show that there exists an arc $\sigma$ joining $x$ to $x_0$ satisfying 
\be \label{big0}
l\left(\sigma[x, z]\right) \leq c' d(z)\text{ for every }z\in \sigma.
\ee 
Set $\alpha = \gamma[x, x_0].$
We divide the proof into two cases.\\
Case 1: $2d(x) \geq d(x_0)$.
\\
Clearly, 
\be \label{big2}
\frac{d(x_0)}{2} \leq  d(x) \leq d(z) + |z-x| \leq (1+c) d(z) \text{ for every } z\in \alpha.
\ee
From the naturality of $\gamma$, \eqref{big1} and \eqref{big2}, it follows that 
\be \label{big3}
k(x,x_0) \leq \diam_k(\alpha) \leq \psi\left(\frac{\diam(\alpha)}{\dist(\alpha, \delta G)}\right) \leq \psi(2c(1+c)).
\ee
Hence there exists a curve $\sigma$ joining $x$ and $x_0$ such that 
$$l_k(\sigma) \leq 2k(x,x_0) \leq 2\psi(2c(1+c)):= A$$
Then for all $z\in \sigma $ we obtain
$$\left| \log \left(\frac{d(z)}{d(x_0)}\right)\right| \leq k(z,x_0) \leq l_k(\sigma) \leq A.$$
Therefore, 
$$ e^{-A} d(x_0) \leq d(z) \leq e^{A} d(x_0) \text{ for every } z\in \sigma$$
which together with \eqref{qh2} implies that
\be \label{big4}
l(\sigma) \leq e^{l_k(\sigma)} d(x_0) \leq e^{A}d(x_0) \leq e^{2A}d(z) \text{ for every } z\in \sigma
\ee
which is the desired result \eqref{big0}.\\[2mm]
Case 2: $2d(x) < d(x_0)$.\\
Let $m$ be the least integer with $m\geq \log_2\left(\frac{d(x_0)}{d(x)}\right)$. \\
Our aim is to find points $x=x_1, x_2,...,x_m,x_{m+1}=x_0$ in the curve $\alpha$ and construct arcs $\beta_i$ between them having uniformly bounded quasihyperbolic length.
Let $x_2\in \alpha $ be the first point such that $d(x_2)= 2d(x_1) <d(x_0)$ while travelling through the arc from $x$ to $x_0$.\\[2mm]
Now for every $z\in \alpha[x_1, x_2]$, we obtain
\be \label{big5}
d(x_1) \leq d(z)+|x_1 -z| \leq (1+c)d(z).
\ee
It is easy to see that,
\be \label{big6}
\diam\left(\alpha[x_1, x_2]\right) \leq cd(x_2) = 2cd(x_1)
\ee
From \eqref{big5}, \eqref{big6} and by the naturality of the curve $\gamma$, we obtain
$$ k(x_1, x_2) \leq \diam_k(\alpha[x_1, x_2]) \leq \psi\left(\frac{\diam(\alpha[x_1, x_2])}{\dist(\alpha[x_1, x_2], \delta G)}\right)\leq \psi (2c(1+c)) .$$
Therefore, there exists a curve $\sigma_1$ joining $x_1$ and $x_2$ such that 
$$l_k(\sigma_1)\leq 2k(x_1, x_2) \leq 2\psi (2c(1+c)):=A.$$
Therefore, a similar calculation as in Case 1, we obtain
$$
\left|\log\left(\frac{d(z)}{d(x_1)}\right)\right| \leq k(z, x_1) \leq l_k(\sigma_1) \leq A \text{ for every } z\in \sigma_1$$
which yields
$$ l(\sigma_1) \leq e^A d(x_1).$$ \\
Now if $2d(x_2) \geq d(x_0) $, we construct a curve $\sigma_2$ joining $x_2$ and $x_3=x_0$ as in Case 1 with $l(\sigma_2)\leq e^A d(x_2)$. 
Otherwise, we continue this process letting $x_3 \in \alpha $ be the first point such that $d(x_3)=2d(x_2)$.
 Since any subarc of a diameter $c-$John arc is again a diameter $c-$John arc a similar calculation  as before yields us an arc $\sigma_2$ joining $x_2$ and $x_3$ with $l(\sigma_2)\leq e^A d(x_2)$. 
We continue this process by letting $x_{i+1}$ as the first point of $\alpha$ with $d(x_{i+1})=2d(x_i)$ and thereby obtain an arc $\sigma_i$ joining $x_i$ to $x_{i+1}$ with \be \label{big7}
l(\sigma_i)\leq  e^Ad(x_i).
\ee
Since, $d(x_i)= 2^{i-1} d(x_1)$ the process will end as soon as $i \geq \log_2\left(d(x_0)/d(x_1)\right)$.
Let $x_m \in \alpha $ be the point such that $2d(x_m) \geq d(x_0) = d(x_{m+1})$. Also 
$$\diam (\alpha[x_m, x_0]) \leq c d(x_0) \leq 2c d(x_m)$$
and 
$$ d(x_m) \leq d(z) + |x_m - z| \leq (1+c) d(z).$$
Therefore, we have 
$$k(x_m, x_0) \leq \psi(2c(1+c)).$$
Choose an arc $\sigma_m$ joining $x_m$ to $x_{m+1}$ satisfying
\be
l_k(\sigma_m)\leq  2\psi(2c(1+c)).
\ee
Similarly as in Case 1 we obtain 
\be  \label{big8}
l(\sigma_m)\leq  e^Ad(x_m).
\ee
We claim that $\sigma = \cup_{i=1}^m \sigma_i$ is a cone arc.\\
Let $z\in \sigma $ then there is $1\leq j \leq m$ so that $z\in \sigma_j$. 
Further,
$$\left| \log\left(\frac{d(x_j)}{d(z)}\right)\right| \leq k(z,x_j) \leq l_k(\sigma_j) \leq A.$$
Therefore, we obtain, 
\be \label{big9}
d(x_j) \leq e^Ad(z).
\ee
From \eqref{big7}, \eqref{big8} and \eqref{big9} it follows that 
$$l\left(\sigma[x,z]\right) \leq \sum_{i=1}^j l(\sigma_j) \leq e^{A} \sum_{i=1}^j d(x_j) \leq 2e^{A} d(x_j) \leq 2e^{2A} d(z),$$
which implies \eqref{big0}.
This completes the proof.

\end{proof}

\begin{proof}[Proof of Corollary \ref{cor1}]
By Theorem \ref{thm2} we obtain that $G'$ is diameter $c'-$John domain with $c'(c, \eta)$. Since $G'$ is $\psi-$natural every diameter $c'-$John curve satisfies the natural condition. Hence by Theorem \ref{thm5} it follows that $G'$ is $c''-$John where $c''=c''(c', \psi)$.
\end{proof}
Now, we are ready to prove Theorem \ref{thm6}.
\begin{proof}[Proof of Theorem \ref{thm6}.]
Let $x'=f(x)$ and $y'=f(y)$ be any two points in $G'$. Since, $x,y \in G$ there exists a $c-$cone arc $\alpha$ joining $x$ and $y$. Clearly, $\alpha $ is a diameter $c-$John arc and by Theorem \ref{thm5} $\alpha $ satisfies the $\psi-$natural condition for some increasing function $\psi :[0, \infty) \rightarrow [0, \infty)$. 
By Theorem \ref{thm5} it is enough to prove that $\alpha' = f(\alpha)$ is a diameter $c'-$ John arc and it satisfies $\psi'-$natural condition for some increasing function $\psi':[0, \infty) \rightarrow [0, \infty)$.\\[2mm]
We see from Remark \ref{remark} that $\alpha'$ is diameter $c'-$John. 
Before we prove that $\alpha'$ satisfies the natural condition we claim that for any connected set $A' \subset G'$ with $\dist(A' \delta G')>0$, we have
\be \label{nat3}
\frac{\diam (A)}{\dist(A, \delta G)} \leq 6 \mu' \left(\frac{\diam (A')}{\dist(A', \delta G')}\right)
\ee
where $A' = f(A)$, $\mu'(t)=1+\eta'(1+t)$ and $\eta'(t) = \left(\eta^{-1}(t^{-1}\right)^{-1}$.\\
Let $x\in A$ and $z\in \delta G$ be two points such that $|x-z| \leq 2\dist\left(A,\delta G\right)$. Choose $y\in A$ such that $\diam(A)\leq 3|x-y| $. 
Since $f^{-1}$ is $\eta'$-quasisymmetric relative to $\delta G'$ with $\eta'(t) = \left(\eta^{-1}(t^{-1}\right)^{-1}$ we obtain
$$ \frac{\diam (A)}{\dist(A, \delta G)} \leq 6 \frac{|x-y|}{|x-z|} \leq 6\left(1+ \eta'\left(\frac{|y-z|}{|x-z|}\right)\right) \leq 6\left(1+ \eta'\left(1+\frac{\diam (A')}{\dist(A', \delta G')}\right)\right), $$
which is the desired result.
It follows from \eqref{nat3} and from the fact that $f$ is $(M, C)-$CQH that for $u', v' \in \alpha'$, we have
\begin{align*}
\diam_{k'}(\alpha'[u', v']) &\leq M \diam_k(\alpha[u, v]) + C \\
& \leq M \psi\left(\frac{\diam (\alpha[u, v])}{\dist(\alpha[u, v], \delta G)}\right) +C \\
& \leq M \psi\left(6\mu'\left(\frac{\diam (\alpha'[u', v'])}{\dist(\alpha[u', v'], \delta G')}\right)\right) +C .
\end{align*} 
The proof is complete by letting $\psi'(t) = M \psi (6\mu'(t))+C$.
\end{proof}
{\bf Acknowledgement.}
The first author thanks SERB-CRG and the second author's research work is supported by CSIR-UGC.

\end{document}